\PassOptionsToPackage{backref=page}{hyperref}
\documentclass[12pt,letterpaper,reqno]{amsart}

\usepackage[T1]{fontenc}
\usepackage[utf8]{inputenc}
\usepackage{amsmath,amssymb,amsfonts,amsthm}
\usepackage{mathtools}
\usepackage{aliascnt}
\usepackage{microtype}
\usepackage{enumitem}
\usepackage{booktabs}
\usepackage{xcolor}
\usepackage{doi}
\usepackage{hyperref}
\usepackage{tikz}
\usetikzlibrary{positioning,arrows.meta}
\usepackage{bookmark}
\usepackage[capitalize,noabbrev]{cleveref}

\hypersetup{
  pdfstartview={FitH},
  colorlinks=true,
  linkcolor=blue!55!black,
  citecolor=green!35!black,
  urlcolor=blue!55!black
}
\renewcommand*{\backref}[1]{}
\renewcommand*{\backrefalt}[4]{%
  \ifcase #1
  \or
    \quad$\hookleftarrow$, cited on page~#2%
  \else
    \quad$\hookleftarrow$, cited on pages~#2%
  \fi
}

\newtheorem{theorem}{Theorem}[section]

\newaliascnt{lemma}{theorem}
\newtheorem{lemma}[lemma]{Lemma}
\aliascntresetthe{lemma}

\newaliascnt{proposition}{theorem}
\newtheorem{proposition}[proposition]{Proposition}
\aliascntresetthe{proposition}

\newaliascnt{corollary}{theorem}

\aliascntresetthe{corollary}

\newaliascnt{conjecture}{theorem}

\aliascntresetthe{conjecture}

\newaliascnt{problem}{theorem}
\newtheorem{problem}[problem]{Problem}
\aliascntresetthe{problem}

\theoremstyle{definition}
\newaliascnt{definition}{theorem}

\aliascntresetthe{definition}

\newaliascnt{example}{theorem}

\aliascntresetthe{example}

\newaliascnt{remark}{theorem}

\aliascntresetthe{remark}

\crefname{theorem}{Theorem}{Theorems}
\Crefname{theorem}{Theorem}{Theorems}
\crefname{lemma}{Lemma}{Lemmas}
\Crefname{lemma}{Lemma}{Lemmas}
\crefname{proposition}{Proposition}{Propositions}
\Crefname{proposition}{Proposition}{Propositions}
\crefname{corollary}{Corollary}{Corollaries}
\Crefname{corollary}{Corollary}{Corollaries}
\crefname{conjecture}{Conjecture}{Conjectures}
\Crefname{conjecture}{Conjecture}{Conjectures}

\newcommand{\C}{\mathbb C}
\newcommand{\R}{\mathbb R}
\newcommand{\Hplus}{\mathbb H^{+}}
\newcommand{\Hminus}{\mathbb H^{-}}
\newcommand{\Chat}{\widehat{\mathbb C}}
\newcommand{\Rhat}{\widehat{\mathbb R}}
\newcommand{\PSL}{\operatorname{PSL}_{2}(\mathbb Z)}
\newcommand{\dist}{\operatorname{dist}}

\begin{document}

\title[Three omitted values and non-Blaschke point divisors]
{Three omitted values and non-Blaschke point divisors in half-planes}

\author[Q.~Tang]{Quanyu Tang$^{*}$}
\thanks{$^{*}$Corresponding author. Email: \texttt{tangquanyu@stu.xjtu.edu.cn}.}
\address{(Q.~Tang) School of Mathematics and Statistics, Xi'an Jiaotong University, Xi'an 710049, P.~R.~China,}

\author[B.~Cui]{Bokai Cui}
\address{(B.~Cui) School of Mathematics and Statistics, Xi'an Jiaotong University, Xi'an 710049, P.~R.~China}

\author[W.~He]{Wei He}
\address{(W.~He) School of Mathematical Sciences, Chongqing Normal University, Chongqing 401331, P.~R.~China}

\author[T.~Hu]{Tao Hu}
\address{(T.~Hu) School of Mathematics and Statistics, Xi'an Jiaotong University, Xi'an 710049, P.~R.~China,}

\author[Y.~Li]{Yanyang Li}
\address{(Y.~Li) School of Mathematics, Southeast University, Nanjing 211189, P.~R.~China}

\author[K.~Wang]{Ke Wang}
\address{(K.~Wang) School of Nuclear Science and Technology, Xi'an Jiaotong University, Xi'an 710049, P.~R.~China}

\author[Z.~Yu]{Zijun Yu}
\address{(Z.~Yu) School of Mathematics and Statistics, Xi'an Jiaotong University, Xi'an 710049, P.~R.~China}


\begin{abstract}
We construct a real meromorphic function $F$ on $\mathbb C$ such that $F^{-1}(\{0,1,\infty\})\subset\mathbb R$, while $F$ is not of bounded type in either half-plane. More strongly, for every $a\in\widehat{\mathbb C}\setminus\{0,1,\infty\}$, the $a$-point divisor in either half-plane fails the Blaschke condition. Thus the construction gives a negative answer to the half-plane bounded-type question arising from Nevanlinna's 1925 work. Postcomposition gives the analogous counterexample for any prescribed triple of distinct values in the Riemann sphere. The core construction and proof were generated during an autonomous run of GPT-5.6 Sol Ultra.
\end{abstract}

\maketitle

\section{Introduction}

A meromorphic function in a domain is said to be of \emph{bounded type} if it
is a quotient of two bounded holomorphic functions; equivalently, it is
meromorphic of bounded characteristic.  For $x>0$, write
$\log^{+}x=\max\{0,\log x\}$ and $\log^{-}x=\max\{0,-\log x\}$.  In his study
of value distribution in angular domains, Nevanlinna introduced half-plane
characteristics and proved, in particular, that a meromorphic function of
finite order on the plane is of bounded type in a half-plane whenever it
omits three distinct values of the Riemann sphere there
\cite{Nevanlinna1925}.  Ostrovskii weakened the finite-order assumption to
\[
 \int_1^\infty \frac{\log^{+}T(r,F)}{r^2}\,dr<\infty,
\]
where $T(r,F)$ is the usual Nevanlinna characteristic
\cite{Ostrovskii1961}; see also
\cite[Chapter~III, Theorems~3.1 and~3.3]{GoldbergOstrovskii2008}.
Gol'dberg subsequently showed that the unrestricted half-plane analogue of
the logarithmic-derivative lemma used by the classical argument is false,
even with an arbitrarily prescribed divergent error term
\cite{Goldberg1975}.

Without the assumption of meromorphic continuation to the whole plane, the
conclusion fails.  Lehto and Virtanen observed that the classical elliptic
modular function is meromorphic in a half-plane, omits three values there,
and is not of bounded type \cite[Section~18]{LehtoVirtanen1957}.  More recently, Eremenko, Kulikov, and Sodin obtained a sharp result for
functions defined in variable neighborhoods of a half-plane
\cite{EremenkoKulikovSodin2026}.  For an even, positive, continuous function
$m$ that is nonincreasing on $[0,\infty)$, they considered
$\mathbb H(-m)=\{x+iy:y>-m(x)\}$ and proved that convergence of
\[
 \int_1^\infty\frac{\log^{-}m(t)}{t^2}\,dt
\]
forces bounded type in $\Hplus$, whereas divergence permits a counterexample
on $\mathbb H(-m)$.

The unrestricted question is a century-old problem going back to
Nevanlinna's 1925 work on meromorphic functions in angular domains
\cite{Nevanlinna1925}.  Nevanlinna proved that the answer is affirmative
under the additional assumption of finite order, but the unrestricted case
remained open for over a century.  Eremenko, Kulikov, and Sodin
recently returned to this question and explicitly recorded it in the
following form \cite[Question~2]{EremenkoKulikovSodin2026}.

\begin{problem}[Nevanlinna's century-old half-plane problem]\label{prob:nevanlinna}
Let $F$ be meromorphic on $\C$.  If $F$ omits three distinct values of
$\Chat$ in a half-plane, must $F$ be of bounded type in that half-plane?
\end{problem}

There is a related history for meromorphic functions whose point divisors are
confined to lines or rays, beginning with Edrei's work
\cite{Edrei1955}.  In the precise line case, Bergweiler and Eremenko proved
the following dichotomy: if the complete preimages of three distinct values
lie on $\R$, then either the real axis is mapped into a circle, or the function
has the explicit exponential-quotient form described in
\cite{BergweilerEremenko2009}; in the latter case the function has order at
most one.  The finite-order conclusion therefore does not apply to the
real-meromorphic branch in which the construction below lies.

We use the notation $\Hplus=\{z:\operatorname{Im}z>0\}$,
$\Hminus=\{z:\operatorname{Im}z<0\}$, $\Chat=\C\cup\{\infty\}$, and
$\Rhat=\R\cup\{\infty\}$.  A meromorphic function $F$ on $\C$
is called \emph{real} if $F(\overline z)=\overline{F(z)}$.  For
$a\in\C$, the $a$-point divisor of $F$ consists of the zeros of $F-a$,
with their multiplicities; for $a=\infty$, it is the pole divisor.  A discrete
divisor $Z$ in either half-plane satisfies the \emph{Blaschke condition} if
\[
 \sum_{z\in Z}\frac{|\operatorname{Im}z|}{1+|z|^2}<\infty,
\]
with multiplicities counted.  We write $\dist$ for Euclidean distance.  Our
main result is as follows.

\begin{theorem}\label{thm:main}
There exists a real meromorphic function $F$ on $\C$ such that
\[
 F^{-1}\bigl(\{0,1,\infty\}\bigr)\subset\R,
\]
and each of the three sets $F^{-1}(0)$, $F^{-1}(1)$, and
$F^{-1}(\infty)$ is infinite.  Moreover, for every
$a\in\Chat\setminus\{0,1,\infty\}$, the $a$-point divisor of $F$ in each of
$\Hplus$ and $\Hminus$ fails the Blaschke condition.  In particular, $F$ is
not of bounded type in either half-plane.
\end{theorem}

The class of functions of bounded type is invariant under postcomposition by
a M\"obius transformation: if $F=U/V$ and
$M(w)=(\alpha w+\beta)/(\gamma w+\delta)$, then
$M\circ F=(\alpha U+\beta V)/(\gamma U+\delta V)$, and the converse follows
by applying $M^{-1}$.  Consequently, $\{0,1,\infty\}$ in
Theorem~\ref{thm:main} can be replaced by any prescribed triple of distinct
values of $\Chat$, and the assertion about all remaining point divisors is
preserved.  An affine change of the independent variable gives the analogous
statement for every Euclidean half-plane.  Thus Theorem~\ref{thm:main}
answers Problem~\ref{prob:nevanlinna} in the negative.

The ingredients of the construction have distinct precedents.  The modular
lambda function and $\Gamma(2)$-orbit counting already enter the local
examples of \cite{EremenkoKulikovSodin2026}.  Conformal mapping followed by
Schwarz reflection is central to the classical theory of comb functions; see
\cite{EremenkoYuditskii2012} and the references therein.  Farey
tessellations and reflection constructions also occur in the study of
conformally balanced trees \cite{IvriiLinRohdeSygal2023}.  The point needed
here is to choose a proper Farey-edge boundary below a quantitatively
controlled graph.  This simultaneously permits meromorphic reflection across
the whole real axis and retains enough Green-function mass to violate the
Blaschke condition.

\subsection{Related work}

He and Zhang~\cite{HeZhang2026} constructed a counterexample to
Problem~\ref{prob:nevanlinna}. Their argument and the construction
presented here share a common framework: both start from the modular
covering, construct a domain bounded by a locally finite chain of
Farey edges, uniformize that domain, and use Schwarz reflection to
obtain a meromorphic function on the whole plane.

The quantitative arguments differ. He and Zhang use finite Farey
refinements in disjoint bays together with a diagonal argument for
divergent Green-weighted spherical area. In the present construction,
a Farey-edge boundary lies below an explicitly controlled
Gaussian-decaying comparison graph. Explicit $\Gamma(2)$-orbits of
individual $a$-points, a Green-function lower bound, and totient
counting then yield the fiberwise conclusion of
Theorem~\ref{thm:main}: every
$a\in\Chat\setminus\{0,1,\infty\}$ has a non-Blaschke $a$-point divisor
in each half-plane.

\subsection{Formal verification}\label{subsec:formal-verification}

Theorem~\ref{thm:main} has been formalized in Lean~4 using Mathlib.
The formalization covers the full statement, including the three
infinite exceptional fibres on the real axis, failure of the Blaschke
condition with multiplicities for every remaining value in both
half-planes, and the two non-bounded-type conclusions.

The formalization is registered with the
\href{https://palomar-registry.org/entry?id=PALOMAR-2026-10-29-000014\&version=1}
{Palomar Registry as \texttt{PALOMAR-2026-10-29-000014}}.
The registered
\href{https://github.com/QuanyuTang/three-omitted-values-palomar/tree/c408da985b3b0498f6f2df702739673865a8b932}
{source snapshot}
is publicly available. The formal proof has been accepted by the
Lean, NanoDa, and con-ron verification kernels, using only the
standard axioms \texttt{propext}, \texttt{Classical.choice}, and
\texttt{Quot.sound}.

\subsection{Declaration of AI usage}

The core mathematical construction and proof underlying this paper
were generated during
\href{https://chatgpt.com/s/cx_6a8c620c99308191972b12d53e3d314f}
{\texttt{an autonomous run}}
of GPT-5.6 Sol using OpenAI's \texttt{ultra} setting.
The run lasted 7 hours, 14 minutes, and 14 seconds.
The resulting candidate proof was subsequently subjected to human
checking, mathematical auditing, and revision. The initial draft
of the manuscript was also generated by AI and was subsequently
reviewed and lightly revised by the authors.

The Lean~4 formalization described in
Subsection~\ref{subsec:formal-verification} was produced using
OpenAI's Codex, with the paper's informal proof and an auxiliary
exposition as source material. The resulting formal proof was
checked by Lean and the independent verification kernels identified
there.

The authors have reviewed the final manuscript and take full
responsibility for all mathematical statements, proofs, conclusions,
and any remaining errors.

\subsection{Paper organization}
The paper is organized as follows.  Section~2 records the modular covering
and the Blaschke criterion used below.  Section~3 constructs a comparison
domain with rapidly decaying boundary and proves a quantitative Green-function
estimate.  In Section~4 we place a proper Farey-edge path below that boundary
and use Schwarz reflection, including a separate analysis at every cusp
(Farey vertex), to obtain a meromorphic function on the plane. Section~5 constructs, inside every remaining point divisor, an explicit
subset that fails the Blaschke condition and completes the proof of
Theorem~\ref{thm:main}.

\section{The modular covering and the Blaschke criterion}

A domain is called \emph{Greenian} if it admits a Green function.  For such a
domain $D$, let $G_D(\zeta,z)$ denote its Green function with pole at
$\zeta$, normalized so that
$G_D(\zeta,z)-\log(1/|z-\zeta|)$ is harmonic near $z=\zeta$.  For a discrete
divisor $Z$ in $\Hplus$, the Blaschke condition above is equivalently
$\sum_{z\in Z}G_{\Hplus}(i,z)<\infty$, where multiplicities are counted
and a possible point at $i$ is omitted.  Every zero divisor of a
nonzero bounded holomorphic function satisfies this condition.  We use these
standard facts in their Green-function form; see, for example,
\cite[Chapter~II]{Garnett2007}.

The Farey tessellation is the ideal triangulation of $\Hplus$ with vertex set
$\mathbb Q\cup\{\infty\}$ in which two reduced fractions $a/b$ and $c/d$
with $|ad-bc|=1$ are joined by the hyperbolic geodesic.  We call these
geodesics Farey edges and, in the modular-covering discussion, refer to the
Farey vertices also as cusps.  Here
$\PSL=\operatorname{SL}_2(\mathbb Z)/\{\pm I\}$, acting by
$\begin{psmallmatrix}a&b\\ c&d\end{psmallmatrix}\!\cdot z=(az+b)/(cz+d)$.

\begin{proposition}\label{prop:modular-cover}
There is a universal covering
$q:\Hplus\to\Chat\setminus\{0,1,\infty\}$ with the following properties.
\begin{enumerate}
\item Its deck group, namely the group of conformal automorphisms $\gamma$ of
$\Hplus$ satisfying $q\circ\gamma=q$, is
\[
 \Gamma(2)=
 \left\{
 \begin{pmatrix}\alpha&\beta\\ \gamma&\delta\end{pmatrix}
 \in\PSL:
 \alpha,\delta\ \text{odd},\quad \beta,\gamma\ \text{even}
 \right\}.
\]
\item The inverse image $q^{-1}(\Rhat)$ is exactly the union of the open
Farey edges in $\Hplus$.
\item The three $\Gamma(2)$-orbits of Farey vertices are represented by the
three nonzero parity classes modulo $2$ of integer vectors with coprime
coordinates.  The limiting value of $q$ at a Farey vertex $r$, denoted by
$s(r)$, is one of $0,1,\infty$ and depends bijectively on this parity class.
\item Let $r$ be a Farey vertex, let $A\in\PSL$ satisfy $A(r)=\infty$, and
put $s=s(r)$.  Define
\[
 \chi_s(\zeta)=
 \begin{cases}
  \zeta-s,&s\in\{0,1\},\\
  1/\zeta,&s=\infty.
 \end{cases}
\]
There is a function $h_{r,A}$, holomorphic near the origin, such that
\begin{equation}\label{eq:cusp-coordinate}
 \chi_s\bigl(q(A^{-1}W)\bigr)=h_{r,A}(e^{\pi iW}),
 \qquad h_{r,A}(0)=0,\qquad h_{r,A}'(0)\ne0,
\end{equation}
whenever $\operatorname{Im}W$ is sufficiently large.
\end{enumerate}
\end{proposition}

\begin{proof}
Take a M\"obius normalization of the classical modular lambda function.
The standard theory of $\lambda$ gives a universal covering of
$\Chat\setminus\{0,1,\infty\}$ with deck group $\Gamma(2)$; it also shows
that the Farey triangles are mapped alternately onto the upper and lower
half-planes, so the real locus is precisely the union of their open edges.
See \cite[Chapter~7, Sections~3.4--3.5, pp.~278--282]{Ahlfors1979},
\cite[Section~23]{Akhiezer1990}, and the
recent use of the same normalization in
\cite[Section~3.3]{EremenkoKulikovSodin2026}.  Reduction modulo $2$
identifies the three cusp orbits with the three nonzero parity classes of
integer vectors with coprime coordinates, giving the third assertion.

It remains only to record the local normalization at a cusp.  The subgroup
$\Gamma(2)$ is normal in $\PSL$, and the stabilizer of $\infty$ in
$\Gamma(2)$ is generated by $W\mapsto W+2$.  Hence $Q=e^{\pi iW}$ is a
primitive local coordinate on the quotient at every cusp, after conjugation
by $A$.  The quotient map is conformal from a punctured cusp neighborhood
onto a punctured neighborhood of $s(r)$.  Therefore
$\chi_s(q(A^{-1}W))$ is a holomorphic function of $Q$ with a simple zero at
$Q=0$, which is exactly \eqref{eq:cusp-coordinate}.
\end{proof}

\section{A comparison domain}

For $z\in\Hplus$, define
\[
 C(z)=\frac1\pi\int_{\R}\frac{e^{-t^2}}{t-z}\,dt.
\]
Inserting $(t-z)^{-1}=i\int_0^\infty e^{-i(t-z)u}\,du$ and evaluating the
Fourier transform of the Gaussian gives
\begin{equation}\label{eq:fourier-C}
 C(z)=\frac{i}{\sqrt\pi}\int_0^\infty e^{-u^2/4}e^{iuz}\,du.
\end{equation}
The representation \eqref{eq:fourier-C}, together with differentiation under
the integral sign, extends $C$ and $C'$ continuously to
$\overline{\Hplus}$.  Consequently,
\[
 |C(z)|\le1,\qquad |C'(z)|\le\frac2{\sqrt\pi}
 \quad(z\in\overline{\Hplus}),
 \qquad \operatorname{Im}C(x)=e^{-x^2}\quad(x\in\R).
\]
The original Cauchy integral gives $\operatorname{Im}C(z)>0$ for
$z\in\Hplus$.

Fix $\varepsilon=10^{-3}$, set $\psi(z)=z+\varepsilon C(z)$, and put
$\delta=2\varepsilon/\sqrt\pi<0.00113$.  Integrating $C'$ along line
segments in $\overline{\Hplus}$ yields
\[
 (1-\delta)|z-w|\le |\psi(z)-\psi(w)|
 \le(1+\delta)|z-w|.
\]
Thus $\psi$ is a bi-Lipschitz homeomorphism of $\overline{\Hplus}$ onto
its image.  Since $|\psi(z)-z|\le\varepsilon$, this homeomorphism is
proper, in the sense that preimages of compact sets are compact.  On the
boundary, $\psi(t)=X(t)+i\varepsilon e^{-t^2}$, where
$X(t)=t+\varepsilon\operatorname{Re}C(t)$.  Since
$X'(t)\ge1-\delta$ and $|X(t)-t|\le\varepsilon$, the map $X$ is an
increasing homeomorphism of $\R$ onto $\R$.  Thus $\psi(\R)$ is the graph of
a positive function $\beta$.  Properness shows that
$\partial\psi(\Hplus)=\psi(\R)$, so $\psi(\Hplus)$ is one of the two
components of the complement of this graph.  Since
$\psi(iy)=iy+O(1)$ as $y\to\infty$, it is the upper component; hence
\[
 \Omega_0:=\psi(\Hplus)=\{x+iy:y>\beta(x)\}.
\]
The estimates $|X(t)-t|\le\varepsilon$ and
$t^2\ge X(t)^2/2-\varepsilon^2$ give
\begin{equation}\label{eq:beta-estimates}
 0<\beta(x)\le
 \varepsilon e^{\varepsilon^2}e^{-x^2/2},
 \qquad \|\beta'\|_\infty<10^{-3}.
\end{equation}
Indeed,
$|\beta'(X(t))|\le\varepsilon\sqrt{2/e}/(1-\delta)<10^{-3}$.

Let $p=\psi(i)$.  The following estimate is the only quantitative property
of $\Omega_0$ that will be needed.

\begin{lemma}\label{lem:green-comparison}
If $w=x+iy\in\Omega_0$, $0<y\le1$, and $\beta(x)\le y/2$, then
\begin{equation}\label{eq:green-comparison}
 G_{\Omega_0}(p,w)\ge\frac{y}{12(1+x^2)}.
\end{equation}
\end{lemma}

\begin{proof}
Write $\zeta=u+iv=\psi^{-1}(w)$ and let
$L=\|\beta'\|_\infty<10^{-3}$.  The vertical gap
$y-\beta(x)$ is at least $y/2$, and hence
\[
 \dist(w,\partial\Omega_0)\ge\frac{y}{2\sqrt{1+L^2}}.
\]
The upper bi-Lipschitz estimate, applied to the distance from $\zeta$ to
$\R$, gives $\dist(w,\partial\Omega_0)\le(1+\delta)v$.  Therefore
$v>y/3$.  Moreover, $v<y\le1$, because
$y=v+\varepsilon\operatorname{Im}C(\zeta)$, and
$|u-x|\le\varepsilon$.  Conformal invariance and the explicit Green function
of $\Hplus$ give
\begin{align*}
 G_{\Omega_0}(p,w)
 &=G_{\Hplus}(i,u+iv)\\
 &=\frac12\log\frac{u^2+(v+1)^2}{u^2+(v-1)^2}\\
 &\ge\frac{2v}{u^2+(1+v)^2}
 \ge\frac{v}{2(1+u^2)}.
\end{align*}
Here the first inequality follows from
$\log(1+t)\ge t/(1+t)$ for $t\ge0$, and the second uses $v\le1$.
Since $v>y/3$ and $1+u^2\le2(1+x^2)$, this proves
\eqref{eq:green-comparison}.
\end{proof}

\section{A Farey-edge boundary and Schwarz reflection}

For every $n\in\mathbb Z$, let
$\mu_n=\min_{x\in[n,n+1]}\beta(x)>0$, and choose an integer $L_n\ge2$ such
that $(2L_n)^{-1}<\mu_n$.  Recall that the Farey sequence of order $L$ in
$[0,1]$ is the increasing sequence of reduced fractions in $[0,1]$ whose
denominators are at most $L$.  Translate the Farey sequence of order $L_n$
to $[n,n+1]$.  If $a/b<c/d$ are consecutive terms of this sequence, then
$bc-ad=1$ and $b+d>L_n$; these standard Farey-sequence facts
may be found, for example, in \cite[Chapter~III]{HardyWright2008}.  Since
$b,d\ge1$, we have $bd\ge b+d-1\ge L_n$, so the geodesic semicircle joining
them has Euclidean height $1/(2bd)\le(2L_n)^{-1}<\mu_n$.

After the duplicated integer endpoints are identified, the translated
sequences form a strictly increasing, locally finite, doubly infinite
sequence of Farey vertices $(r_j)_{j\in\mathbb Z}$.  Let $\mathcal P$ be the
union of the Farey edges joining $r_j$ to $r_{j+1}$.  It is the graph of a
continuous function $h_{\mathcal P}$, and the inclusion of $\mathcal P$ in
$\C$ is proper, meaning that compact subsets of $\C$ meet only compact
subarcs of $\mathcal P$.  By construction,
\begin{equation}\label{eq:path-below}
 0\le h_{\mathcal P}(x)<\beta(x)\qquad(x\in\R).
\end{equation}
Set $\Omega=\{x+iy:y>h_{\mathcal P}(x)\}$.  By
\eqref{eq:path-below},
\begin{equation}\label{eq:domain-inclusions}
 \Omega_0\subset\Omega\subset\Hplus.
\end{equation}
Figure~\ref{fig:domains} shows the geometry of these two domains.

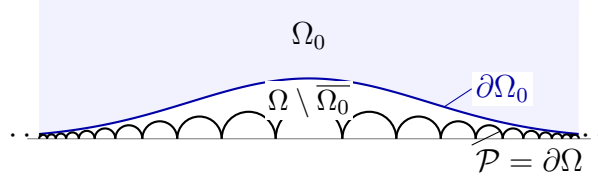
\begin{figure}[tbp]
\centering
\begin{tikzpicture}[x=1.05cm,y=1.05cm]

  % A schematic rapidly decaying comparison boundary
  \def\betaexpr{0.76*exp(-0.22*\x*\x)}

  % Shade Omega_0
  \fill[blue!5]
    plot[smooth,domain=-3.4:3.4,samples=160]
      (\x,{\betaexpr})
    -- (3.4,1.75) -- (-3.4,1.75) -- cycle;

  % Boundary of Omega_0
  \draw[blue!65!black,thick]
    plot[smooth,domain=-3.4:3.4,samples=160]
      (\x,{\betaexpr});

  % Real axis
  \draw[gray] (-3.55,0)--(3.55,0);

  % Macro for a semicircular Farey edge
  \newcommand{\fareyarc}[2]{%
    \draw[black,thick]
      (#1,0) arc[start angle=180,end angle=0,radius=#2];
  }

  % Schematic Farey-edge path.
  % The arcs decrease toward both ends and remain below
  % the comparison boundary.
  \fareyarc{-3.40}{0.050}
  \fareyarc{-3.30}{0.055}
  \fareyarc{-3.19}{0.065}
  \fareyarc{-3.06}{0.080}
  \fareyarc{-2.90}{0.100}
  \fareyarc{-2.70}{0.130}
  \fareyarc{-2.44}{0.170}
  \fareyarc{-2.10}{0.220}
  \fareyarc{-1.66}{0.280}
  \fareyarc{-1.10}{0.340}
  \fareyarc{-0.42}{0.420}
  \fareyarc{ 0.42}{0.340}
  \fareyarc{ 1.10}{0.280}
  \fareyarc{ 1.66}{0.220}
  \fareyarc{ 2.10}{0.170}
  \fareyarc{ 2.44}{0.130}
  \fareyarc{ 2.70}{0.100}
  \fareyarc{ 2.90}{0.080}
  \fareyarc{ 3.06}{0.065}
  \fareyarc{ 3.19}{0.055}
  \fareyarc{ 3.30}{0.050}

  % Domain labels
  \node[fill=blue!5,inner sep=1.5pt]
    at (0,1.30) {$\Omega_0$};

  \node[fill=white,inner sep=1.1pt,font=\small]
    at (0,0.48)
    {$\Omega\setminus\overline{\Omega_0}$};

  % Label for the comparison boundary
  \node[
    blue!65!black,
    fill=white,
    inner sep=1.3pt,
    anchor=west
  ] (omegaLabel) at (2.05,0.63)
    {$\partial\Omega_0$};

  \draw[blue!65!black,thin]
    (omegaLabel.west) -- (1.70,0.40);

  % Label for the Farey-edge boundary
  \node[
    fill=white,
    inner sep=1.3pt,
    anchor=north west
  ] (pLabel) at (2.05,-0.08)
    {$\mathcal P=\partial\Omega$};

  \draw[thin]
    (pLabel.north west) -- (2.42,0.10);

  % Indicate continuation toward infinity
  \node at (-3.52,0.035) {$\cdots$};
  \node at (3.52,0.035) {$\cdots$};

\end{tikzpicture}
\caption{The comparison domain $\Omega_0$ and a schematic portion of the
proper Farey-edge boundary $\mathcal P=\partial\Omega$.  The boundary
$\partial\Omega_0$ approaches the real axis, while
$\mathcal P$ remains strictly below $\partial\Omega_0$.}
\label{fig:domains}
\end{figure}

The one-point compactification $\mathcal P\cup\{\infty\}$ is a Jordan curve.
Indeed, $\mathcal P$ is a proper embedding of $\R$, and its real coordinate
tends to $-\infty$ and $+\infty$ at its two ends.  Hence $\Omega$ is a Jordan
domain in the sphere.  By the Riemann mapping theorem and the Carath\'eodory
boundary theorem, there is a conformal map
\begin{equation}\label{eq:Riemann-map}
 \Phi:\Hplus\longrightarrow\Omega,
 \qquad \Phi(i)=p,\qquad \Phi(\infty)=\infty,
\end{equation}
where the last equality refers to the boundary prime end represented by
curves tending to infinity.  Since $\Omega$ is a Jordan domain, its prime
ends are canonically identified with its boundary points, so the prime-end
closure may here be identified with the ordinary closure in the Riemann
sphere.  Moreover, $\Phi$ extends to a homeomorphism of the prime-end
closures; see, for example,
\cite[Chapters~2 and~3]{Pommerenke1992}.

\begin{proposition}\label{prop:global-reflection}
With $\Phi$ as in \eqref{eq:Riemann-map}, the function $f=q\circ\Phi$ on
$\Hplus$ extends to a real meromorphic function $F$ on $\C$.  Its preimages of
$0,1,\infty$ are real, and each of these three point sets is infinite.
\end{proposition}

\begin{proof}
The boundary extension of $\Phi$ and Proposition~\ref{prop:modular-cover}
show that $f$ has real boundary values on the inverse image of every open
edge of $\mathcal P$.  Schwarz reflection therefore extends $f$
holomorphically across the corresponding open interval of $\R$.  It remains
to treat the boundary points that map to Farey vertices.

Fix a finite vertex $r$ of $\mathcal P$, and let $x_r\in\R$ be determined by
$\Phi(x_r)=r$.  Write $r=p_0/q_0$ in lowest terms with $q_0>0$.  Choose
$a_0,b_0\in\mathbb Z$ such that $a_0p_0+b_0q_0=-1$, and set
\[
 A(W)=\frac{a_0W+b_0}{q_0W-p_0}\in\PSL.
\]
Then $A(r)=\infty$. If $u=c_0/d_0$ is a Farey neighbor of $r$, with $d_0>0$, then
\[
 A(u)-\frac{a_0}{q_0}
 =-\frac{d_0}{q_0(q_0c_0-p_0d_0)}.
\]
For the two neighbors of $r$ occurring immediately to its left and right on
$\mathcal P$, the quantities $q_0c_0-p_0d_0$ are respectively $-1$ and $1$.
Thus the two incident Farey edges are carried by $A$ to vertical geodesics on
opposite sides of $a_0/q_0$.  Their finite endpoints are integers, so, after
interchanging them if necessary, these geodesics have equations
$\operatorname{Re}W=k$ and $\operatorname{Re}W=k+N$, where
$k\in\mathbb Z$ and $N\ge1$.  Moreover,
\[
 A(r+iy)=\frac{a_0}{q_0}+\frac{i}{q_0^2y}\qquad(y>0),
\]
which shows explicitly that the side of $\mathcal P$ belonging to $\Omega$
is mapped, near the cusp, to the strip between these two vertical lines.
The complement in $\mathcal P\cup\{\infty\}$ of sufficiently short subarcs
of the two incident edges is compact and does not contain $r$; its image
under $A$ is therefore bounded.  Hence one can choose $Y$ so large that
\[
 A(\Omega)\cap\{\operatorname{Im}W>Y\}
 =\{k<\operatorname{Re}W<k+N,\ \operatorname{Im}W>Y\}.
\]

The map
$\eta(W)=\exp(\bigl(\pi i(W-k)/N\bigr))$ sends this strip end conformally
onto a punctured upper half-disc.  Thus
$\Theta(z)=\eta(A(\Phi(z)))$ is holomorphic in a one-sided neighborhood of
$x_r$ in $\Hplus$, tends to $0$ as $z\to x_r$, and has real boundary values
on both adjacent intervals: the two boundary lines are mapped by $\eta$ to
the positive and negative real axes.  Schwarz reflection extends $\Theta$
holomorphically to a punctured full neighborhood of $x_r$, and the
singularity at $x_r$ is removable because $\Theta$ is bounded there.  Since
$\Theta$ is nonconstant, for some integer $\ell\ge1$ we have
\[
 \Theta(z)=(z-x_r)^\ell g_r(z),
 \qquad g_r(x_r)\ne0,
\]
with $g_r$ holomorphic near $x_r$.  Moreover,
$e^{\pi iA(\Phi(z))}=(-1)^k\Theta(z)^N$.

Let $s=s(r)\in\{0,1,\infty\}$ be the cusp value of $r$.  By
\eqref{eq:cusp-coordinate}, on the original upper half-neighborhood,
\begin{equation}\label{eq:vertex-representation}
 \chi_s(f(z))
 =h_{r,A}\bigl((-1)^k\Theta(z)^N\bigr).
\end{equation}
On either boundary line of the cusp strip, $q(A^{-1}W)$ is real.  It follows
that $h_{r,A}$ takes real values on a nontrivial real interval ending at the
origin, and therefore
$h_{r,A}(\overline\zeta)=\overline{h_{r,A}(\zeta)}$ near $0$.  The right-hand
side of \eqref{eq:vertex-representation} consequently supplies a
single-valued meromorphic continuation of $f$ across $x_r$ that agrees with
the edge reflections on both sides.  Since $h_{r,A}'(0)\ne0$, the point
$x_r$ is a zero, a $1$-point, or a pole according as $s=0$, $s=1$, or
$s=\infty$, and its corresponding order is $N\ell$.

The vertices of $\mathcal P$ are locally finite and accumulate only at the
prime end $\infty$.  Their preimages under the boundary homeomorphism of
$\Phi$ have the same property.  The edge reflections and the vertex
continuations agree on overlaps by the identity theorem.  Together they fill
in the values on $\R$ and extend the piecewise formula
\begin{equation}\label{eq:global-reflection}
 F(z)=
 \begin{cases}
  f(z),&\operatorname{Im}z>0,\\
  \overline{f(\overline z)},&\operatorname{Im}z<0
 \end{cases}
\end{equation}
to one meromorphic function on $\C$.  This function is real.

The covering $q$ omits $0,1,\infty$ in $\Hplus$.  On an open Farey edge it
takes values in $\Rhat\setminus\{0,1,\infty\}$, while at a Farey vertex the
continuation takes the associated cusp value.  Hence
$F^{-1}(\{0,1,\infty\})\subset\R$.  Finally, because $L_n\ge2$, the part of
$\mathcal P$ over $[n,n+1]$ contains the vertices $n$, $n+1/2$, and $n+1$.
Representatives with coprime coordinates for these vertices have, modulo
$2$, the three nonzero parity classes $(0,1)$, $(1,0)$, and $(1,1)$, in an
order depending on the parity of $n$.  Proposition~\ref{prop:modular-cover}
therefore shows that all three cusp values occur over every such interval.
Each of $F^{-1}(0)$, $F^{-1}(1)$, and $F^{-1}(\infty)$ is consequently
infinite.
\end{proof}

\section{Non-Blaschke point divisors}

We now prove the quantitative assertion in Theorem~\ref{thm:main}.  Fix
$a\in\Chat\setminus\{0,1,\infty\}$, and choose a lift
$\tau_a=\xi+i\nu\in\Hplus$ with $q(\tau_a)=a$.

For $m\ge2$, put
$D_m=\{d:1\le d\le2m,\ \gcd(d,2m)=1\}$.  If $d\in D_m$, then
$\gcd(d,4m)=1$, so let $\alpha_{m,d}\in\{1,\ldots,4m-1\}$ be the inverse
of $d$ modulo $4m$, and set
$b_{m,d}=(\alpha_{m,d}d-1)/(2m)$.  Then
$\alpha_{m,d}d-2mb_{m,d}=1$.  Moreover,
$\alpha_{m,d}d\equiv1\pmod{4m}$ shows that $b_{m,d}$ is even, while both
$d$ and $\alpha_{m,d}$ are odd.  Hence
\[
 \gamma_{m,d}=
 \begin{pmatrix}
  \alpha_{m,d}&b_{m,d}\\
  2m&d
 \end{pmatrix}
 \in\Gamma(2).
\]
The determinant identity gives
\begin{equation}\label{eq:orbit-identity}
 \gamma_{m,d}\tau_a
 =\frac{\alpha_{m,d}}{2m}
  -\frac{1}{2m(2m\tau_a+d)}.
\end{equation}

Let $n_m=\lceil4\sqrt{\log m}\rceil$ and
$w_{m,d}=\gamma_{m,d}\tau_a+2n_m$.  Translation by $2n_m$ belongs to
$\Gamma(2)$, so $q(w_{m,d})=a$.  Write
$w_{m,d}=x_{m,d}+iy_{m,d}$.  From \eqref{eq:orbit-identity},
\begin{equation}\label{eq:orbit-height}
 y_{m,d}=\frac{\nu}{(2m\xi+d)^2+4m^2\nu^2}.
\end{equation}
Equations \eqref{eq:orbit-identity} and \eqref{eq:orbit-height} imply that,
uniformly for $d\in D_m$,
\begin{equation}\label{eq:orbit-bounds}
 \frac{c_a}{m^2}\le y_{m,d}\le\frac{C_a}{m^2},
 \qquad x_{m,d}=2n_m+O_a(1),
\end{equation}
where the notation $O_a(1)$ means that the implicit constant depends only on
the chosen lift $\tau_a$.  For example, one may take
$c_a=\nu/[4((|\xi|+1)^2+\nu^2)]$ and $C_a=1/(4\nu)$; the real correction
term in \eqref{eq:orbit-identity} has modulus at most
$(4\nu m^2)^{-1}$.

For all sufficiently large $m$, \eqref{eq:orbit-bounds} gives
$n_m\le x_{m,d}\le3n_m$.  Together with \eqref{eq:beta-estimates}, this
yields
\[
 \beta(x_{m,d})
 \le \varepsilon e^{\varepsilon^2}e^{-n_m^2/2}
 \le \varepsilon e^{\varepsilon^2}m^{-8}
 \le\frac{y_{m,d}}2.
\]
Thus $w_{m,d}\in\Omega_0\subset\Omega$ and $0<y_{m,d}\le1$ once $m$ is
large enough.

The points $w_{m,d}$ are all distinct.  Indeed, write
\[
 \sigma_{m,d}=
 \begin{pmatrix}1&2n_m\\0&1\end{pmatrix}\gamma_{m,d}\in\Gamma(2),
 \qquad w_{m,d}=\sigma_{m,d}\tau_a.
\]
Equality of two selected points would make the quotient of the corresponding
deck transformations fix $\tau_a$.  A deck transformation of a universal
covering acts freely, so the two transformations agree as elements of $\PSL$.
Left multiplication by the translation matrix does not change the lower row
$(2m,d)$.  Since both entries of this lower row are positive, the sign
ambiguity in $\PSL$ is removed, and equality of the transformations forces
equality of the pairs $(m,d)$. Thus level $m$ contains exactly
$\varphi(2m)$ selected points, where $\varphi$ is Euler's totient function,
$\varphi(n)=\#\{1\le k\le n:\gcd(k,n)=1\}$.

Put $z_{m,d}=\Phi^{-1}(w_{m,d})$.  Since both $q$ and $\Phi$ are locally
conformal, these are distinct simple $a$-points of $F$ in $\Hplus$.
Conformal invariance of Green functions, domain monotonicity,
\eqref{eq:domain-inclusions}, and Lemma~\ref{lem:green-comparison}, specifically
\eqref{eq:green-comparison}, give
\begin{align}
 G_{\Hplus}(i,z_{m,d})
 &=G_\Omega(p,w_{m,d})\notag\\
 &\ge G_{\Omega_0}(p,w_{m,d})\notag\\
 &\ge\frac{y_{m,d}}{12(1+x_{m,d}^2)}
 \ge\frac{c'_a}{m^2\log m}                 \label{eq:fiber-green}
\end{align}
for all sufficiently large $m$, where $c'_a>0$ depends only on $\tau_a$.

We now invoke the classical summatory estimate
\begin{equation}\label{eq:summatory-totient}
 \sum_{n\le X}\varphi(n)=\frac{3}{\pi^2}X^2+O(X\log X),
\end{equation}
see \cite[Theorem~3.7]{Apostol1976}.  By
\eqref{eq:summatory-totient}, for some absolute constant $c_0>0$ and all
sufficiently large $N$,
$\sum_{N<m\le2N}\varphi(m)\ge c_0N^2$.  Since
$\varphi(2m)\ge\varphi(m)$, we obtain
\[
 \sum_{N<m\le2N}\frac{\varphi(2m)}{m^2\log m}
 \ge\frac{c_0}{4\log(2N)}.
\]
Summing over the disjoint dyadic blocks $N=2^j$ proves
\begin{equation}\label{eq:totient-divergence}
 \sum_{m=2}^\infty\frac{\varphi(2m)}{m^2\log m}=\infty.
\end{equation}
Combining \eqref{eq:fiber-green} and \eqref{eq:totient-divergence}, and
discarding finitely many initial levels, gives an integer $m_0(a)\ge2$ such
that
\begin{equation}\label{eq:non-blaschke}
 \sum_{m\ge m_0(a)}\ \sum_{d\in D_m}
 G_{\Hplus}(i,z_{m,d})=\infty.
\end{equation}
Since $\operatorname{Re}w_{m,d}\ge n_m\to\infty$ whereas $\Phi(i)=p$ is
fixed, $m_0(a)$ may be enlarged so that none of the selected points is the
base point $i$.  Thus a subset of the $a$-point divisor in $\Hplus$ already
fails the Blaschke condition.  Since $a$ was arbitrary, this holds for every
$a\in\Chat\setminus\{0,1,\infty\}$.

For the lower half-plane, apply the upper-half-plane assertion to
$\overline a$ and reflect the resulting points.  Indeed,
\eqref{eq:global-reflection} gives
$F(\overline z)=\overline{F(z)}$, and
$G_{\Hminus}(-i,\overline z)=G_{\Hplus}(i,z)$.

Finally, suppose that $F$ were of bounded type in $\Hplus$, and write
$F=U/V$ with bounded holomorphic functions $U,V$ and $V\not\equiv0$.
For any fixed $a\in\C\setminus\{0,1\}$, the bounded holomorphic function
$U-aV$ is not identically zero because $F$ is nonconstant, and it vanishes at
every selected simple $a$-point $z_{m,d}$.  Its zero divisor would therefore
have a divergent Green sum by \eqref{eq:non-blaschke}, contradicting the
Blaschke condition for zeros of a bounded holomorphic function.  Hence $F$
is not of bounded type in $\Hplus$.  Real symmetry gives the same conclusion
in $\Hminus$.  Together with Proposition~\ref{prop:global-reflection}, this
completes the proof of Theorem~\ref{thm:main}.

\end{document}